\documentclass{article}
\usepackage{graphicx} 
\usepackage{amsthm}
\usepackage{amsmath}
\usepackage{amsfonts}
\usepackage{amssymb}

\usepackage[hyphens]{url}
\usepackage[hidelinks]{hyperref}
\usepackage[hyphenbreaks]{breakurl}

\newtheorem{defn}{Definition}
\newtheorem{thm}{Theorem}
\newtheorem{prop}{Proposition}
\newtheorem{lem}{Lemma}
\newtheorem{cor}{Corollary}
\theoremstyle{remark}

\newcommand{\Z}{\mathbb{Z}}
\newcommand{\mycomment}[1]{}

\title{On the singularity probability of random circulant Bernoulli matrices}
\author{Niklas Miller}
\date{\today}

\begin{document}

\maketitle

\begin{abstract}
    A complete characterization of the asymptotic singularity probability of random circulant Bernoulli matrices is given for all values of the probability parameter.
\end{abstract}

\section{Introduction}

In random matrix theory, structured random matrices have gained interest in recent years. The spectral properties of random Toeplitz, circulant, Hankel and other structured matrix models exhibit interesting limiting features, in contrast to the situation with classical matrix models, where the limiting spectral distribution is usually a rather dull distribution such as the semi-circular or circular distribution. 

Although many of the asymptotic spectral properties of random structured matrices have been established in the literature before \cite{bose2002limiting,meckes2009some}, a question which has not been answered before to satisfactory precision is the following: what is the asymptotic singularity probability of a random structured matrix? In this paper, we address this question for an important class of random structured matrices: circulant (signed) Bernoulli matrices. 

Circulant matrices were
first introduced in 1846 by Catalan \cite{catalan1846recherches}, and have since then found many applications in mathematics, coding theory, combinatorial design theory, time-series analysis, machine learning, signal processing, image compression, and specifically in applications involving the Discrete Fourier Transform (DFT). They can be used to elegantly derive the solution sets of low-degree polynomials \cite{kalman2001polynomial}. The convolution operation between two vectors can conveniently be represented as a matrix multiplication of a circulant matrix and a vector. Circulant matrices with entries in $\{0,1\}$ or $\{1,-1\}$ appear frequently in coding and communications theory, combinatorics and graph theory.

Let us briefly review what is known about the singularity probability problem for general $n\times n$ signed Bernoulli matrices with i.i.d. entries. Let $p_q(n)$ denote the probability that an $n\times n$ signed Bernoulli matrix $M_n$, where each entry is $1$ with probability $q$ and $-1$ with probability $1-q$, has vanishing determinant. It was first proved in 1967 by Koml\'os \cite{komlos1967determinant} that $\lim_{n\to \infty} p_{\frac{1}{2}}\left(n\right)=0$. The bound $p_{\frac{1}{2}}\left(n\right)\geq \frac{1}{2^n}$ follows by considering the probability that two given rows are equal. It had been conjectured that this is the dominant source of singularity, i.e., that $p_{\frac{1}{2}}\left(n\right)=\left(\frac{1}{2}+o(1)\right)^n$ as $n\to \infty$. The first exponentially decreasing bound on $p_{\frac{1}{2}}\left(n\right)$ was obtained by Kahn, Koml\'os and Szemer\'edi in 1995 \cite{kahn1995probability}: $p_{\frac{1}{2}}\left(n\right)\leq 0.999^n$. Subsequently, gradual improvements to the basis of the exponential were made by Tao and Wu  \cite{tao2005random, tao2007singularity} and by Bourgain, Vu, and Wood \cite{BOURGAIN2010559}. The conjecture was finally proved by Tikhomirov \cite{tikhomirov2020singularity}, who also showed that more generally, $p_q(n)= (1-q+o(1))^n$ for all $q\in (0,\frac{1}{2}]$.

When the entries of a random matrix exhibit dependencies, less is currently understood about the corresponding problem. It was recently shown that for symmetric signed Bernoulli matrices, the singularity probability is exponentially small \cite{campos2024singularity}. For the class of circulant matrices, we introduce the following notation. $P_q(n)$ (respectively, $P_q^{+}(n)$) is the singularity probability of an $n\times n$ circulant matrix whose first row has each entry $1$ with probability $q$ and 0 (respectively, $-1$) with probability $1-q$, independently of each other. The results of Meckes \cite{meckes2009some} imply that $P_{\frac{1}{2}}^{+}\left(n\right)=O\left(\frac{1}{n}\right)$, and there exist absolute constants $c_1,c_2>0$ such that $\frac{c_1}{\sqrt{n}}\leq P_{\frac{1}{2}}^{+}\left(n\right)\leq \frac{c_2}{\sqrt{n}}$ for all \emph{even} $n$. Meckes writes \cite{meckes2009some}: "It would be nice to have a more complete description of the dependence of the singularity probability on the prime factorization of $n$".

The purpose of the present paper to give a complete description of the asymptotic singularity probability of (signed) circulant Bernoulli matrices with parameter $q$, for all values $q\in [0,1]$. More precisely, the next theorem is our main contribution. We let $p:\mathbb{Z}_{\geq 2}\to\mathbb{Z}_{\geq 1}$ be the function which maps $n$ to the smallest prime divisor $p(n)$ of $n$. For functions $f,g:\mathbb{Z}_{\geq 1}\to\mathbb{R}_{>0}$, the asymptotic notation $f\sim g$ means that $\lim_{n\to\infty}\frac{f(n)}{g(n)}=1$. 

\begin{thm}
    \label{thm:main0}
    Let $q\in (0,1)$ be fixed. Then
    \begin{align*}
    P_q(n) &\sim \sum_{k=0}^{\frac{n}{p(n)}} \phi_q(k,n/p(n))^{p(n)}
    \end{align*}
    as $n\to\infty$, where $k\mapsto \phi_q(k,n)=q^k(1-q)^{n-k}\binom{n}{k}$ is the probability density function of the binomial distribution with parameters $n$ and $q$, and $p(n)$ is the smallest prime divisor of $n$.
\end{thm}


The analogous result for $P_q^{+}(n)$ is very similar and can be found in Corollary \ref{cor:main}. Although the determinant of an $n\times n$ circulant matrix is highly dependent on the prime factorization of $n$, Theorem \ref{thm:main0} gives a unified expression for all $n$, and in particular shows that the smallest prime divisor of $n$ determines the singularity probability. Note that when $n$ is prime, $P_q(n)=q^n+(1-q)^n$, so the asymptotic expression in Theorem \ref{thm:main0} is in fact an equality. 

In Proposition \ref{prop:asymptotic2}, we derive an asymptotic expression for the sum appearing in Theorem \ref{thm:main0}, which might have independent value. It generalizes an asymptotic formula of sums of powers of binomial coefficients derived in \cite{farmer2004asymptotic}. Using the proposition, a more transparent version of Theorem \ref{thm:main0} can be stated.

\begin{cor}
    Let $p$ be a fixed prime, and let $q\in (0,1)$ be fixed. Then 
    \begin{align*}
        P_q(n)&\sim\frac{1}{\sqrt{p}}\left(\frac{p}{2\pi q(1-q)}\right)^{\frac{p-1}{2}} \frac{1}{n^{\frac{p-1}{2}}}
    \end{align*}
    as $n\to\infty$, where $n$ is a composite integer whose smallest prime divisor is $p$.
\end{cor}
This shows that once we fix a prime $p$ and $q\in(0,1)$, then for all sufficiently large $n$ whose smallest prime divisor is $p$, the singularity probability decays at the same rate as $\frac{1}{n^{\frac{p-1}{2}}}$. Thus, for even $n$ and $q=\frac{1}{2}$, we obtain
\begin{align*}
P_q(n) \sim \frac{2}{\sqrt{2\pi n}}.
\end{align*}

The paper is organized in the following way. Section \ref{sec:preliminaries} presents the different notions used in the paper. In Section \ref{sec:pq_bounds}, bounds on the probability that a given eigenvalue of a random circulant Bernoulli matrix vanishes are derived, as well as exact singularity probabilities in certain special cases. In Section \ref{sec:mq_bounds}, the asymptotic behavior of sums of powers of $\phi_q(k,n)$ is analyzed. Finally, in Section \ref{sec:main}, the main theorem of the paper is proved.

\section{Preliminaries}
\label{sec:preliminaries}
Let $G$ be a finite abelian group. We define the notion of $G$-matrix as follows.

\begin{defn}
	\label{def:G_mat}
    Let $G$ be a finite abelian group.
	Let $\{x_g:g\in G\}$ be a set of indeterminates. A $G$-matrix $M$ is a matrix indexed by $G\times G$ such that entry $(g,h)$ of $M$ equals $x_{gh^{-1}}$.
\end{defn}

Note that we have to fix some ordering of the elements of $G$ for the indexing. The homogeneous polynomial $\Delta(G)=\det((x_{gh^{-1}})_{g,h\in G})$ of degree $|G|$ in the variables $x_g$ ($g\in G$) is called the \emph{group determinant} of $G$ and it is independent of the ordering of the elements $G$ chosen for the indexing. When $G=C_n$, the cyclic group of order $n$, a $G$-matrix is called a \emph{circulant matrix}, and it has the following form:
\begin{align*}
	\text{circ}(x_0,x_1,\dots,x_{n-1}) &= \begin{pmatrix}
		x_0 & x_1 & \cdots & x_{n-1} \\
		x_{n-1} & x_0 & \cdots & x_{n-2} \\
		\vdots & \vdots & \cdots & \vdots \\
		x_1 & x_2 & \cdots & x_{0} \\
	\end{pmatrix},
\end{align*} 
where the indexing is modulo $n$. The group determinant $\Delta(G)$ of a finite abelian group splits into linear factors over $\mathbb{C}$, as was noted by Dedekind at the end of the 19th century. Recall that a character of a finite abelian group $G$ is a homomorphism $\chi: G\to \mathbb{C}$. The set of characters $\widehat{G}$ forms a group isomorphic to $G$ under the product defined by $(\chi_1\cdot \chi_2)(g)=\chi_1(g)\chi_2(g)$ for all $g\in G$.

\begin{thm}
    \label{thm:Dedekind}
 	Let $G$ be a finite abelian group and let $\widehat{G}$ be the group of characters. Then $$\Delta(G)=\prod_{\chi\in \widehat{G}}\left(\sum_{g\in G} x_g \chi(g) \right).$$
\end{thm}

Specializing to the group $G=C_n$, we see that the determinant of an $n\times n$ circulant matrix $C_n=\text{circ}(c_0,c_1,\dots,c_{n-1})$ is given by 

\begin{align*}
    \det(C_n)=\prod_{j=0}^{n-1} \left(\sum_{k=0}^{n-1} c_k \zeta_n^{k j}\right)
\end{align*}

where $\zeta_n=\exp\left(\frac{2\pi i}{n}\right)$ and $i^2=-1$. Therefore, $\det(C_n)$ vanishes if and only if $\lambda_j=\sum_{k=0}^{n-1} c_k \zeta_{n}^{kj}=0$ for some $j=0,1,\dots,n-1$. 

In this paper, the ring of circulant matrices with entries in $\mathbb{Z}$ is viewed as the quotient ring $R_n=\mathbb{Z}[x]/(x^n-1)$. Thus, $\text{circ}(c_0,c_1,\dots,c_{n-1})$ is singular if and only if the polynomial $f(x)=\sum_{j=0}^{n-1} c_j x^j \in R_n$ satisfies $f(x)\equiv 0\pmod{\Phi_d(x)}$ for some divisor $d$ of $n$, where $\Phi_d(x)$ is the $d$:th cyclotomic polynomial, defined as the minimal polynomial of any primitive complex $d$:th root of $1$. Explicitly, 
\begin{align*}
    \Phi_n(x)=\prod_{\gcd(j,n)=1}(x-\zeta_n^j).
\end{align*}
Alternatively, $\Phi_n(x)$ can be defined recursively via $\Phi_1(x)=x-1$ and 
\begin{align}
\label{eq:recursion}
    \Phi_{np}(x) &= \begin{cases}\Phi_{n}(x^p),& \text{ if $p\mid n$},\\
    \frac{\Phi_{n}(x^p)}{\Phi_n(x)},& \text{ if $p\nmid n$},
    \end{cases}
\end{align}
 where $p$ is prime and $n$ any positive integer. This shows that 
 \begin{align}
 \label{eq:cyclo_prime}
 \Phi_{p^k}(x)=1+x^{p^{k-1}}+\cdots+x^{(p-1)p^{k-1}}    
 \end{align}
 for all $k\geq 1$ and $p$ prime. As a direct consequence of (\ref{eq:recursion}), the cyclotomic polynomial satisfies the following lemma.

\begin{lem}
    \label{lem:cyclotomic_remainder}
    Let $n$ be a positive integer and let $p$ be a prime. Then $\Phi_{np}(x)\equiv0\pmod{p,\Phi_{n}(x)}$.
\end{lem}

Now we introduce the probabilistic notation. Let $n$ be a positive integer and let $c=(c_j)_{j=0}^{n-1}$ be a sequence of i.i.d. Bernoulli distributed (with parameter $q\in (0,1)$) random variables, i.e., $c_j=1$ with probability $q$ and $c_j=0$ with probability $1-q$. Let $f(x)=\sum_{j=0}^{n-1} c_j x^j$ and $f^+(x)=\sum_{j=0}^{n-1} (2c_j-1) x^j$ denote the random polynomials in $R_n$ obtained from the sequence $c$ (the dependence on $c$ is suppressed for notational convenience).

For each divisor $d\mid n$, define the events $\mathcal{P}_{d,n}=\{f(x)\equiv 0\pmod{\Phi_d(x)}\}$ and similarly $\mathcal{P}_{d,n}^+=\{f^+(x)\equiv 0\pmod{\Phi_d(x)}\}$. We are interested in the following probabilities.

\begin{align*}
    P_q(d,n) &= \mathbb{P}[\mathcal{P}_{d,n}],\qquad P_q(n) = \mathbb{P}[\cup_{d\mid n}\mathcal{P}_{d,n}],\\
    P_q^{+}(d,n) &= \mathbb{P}[\mathcal{P}_{d,n}^+], \qquad P_q^+(n)=\mathbb{P}[\cup_{d\mid n} \mathcal{P}_{d,n}^+].
\end{align*}

Note that $P_q(n)$ (respectively, $P_q^+(n)$) is the singularity probability of a binary (respectively, signed) circulant matrix whose entries are (signed) Bernoulli distributed with parameter $q$. For all divisors $d\neq 1$ of $n$, $P_q(d,n)=P_q^+(d,n)$. Indeed, if $\sum_{j=0}^{n-1} c_j x^j\equiv 0\pmod{\Phi_d(x)}$ then
\begin{align*}
    \sum_{j=0}^{n-1}(2c_j-1) x^j= 2\sum_{j=0}^{n-1} c_j x^j -\frac{x^n-1}{x-1}\equiv 0\pmod{\Phi_d(x)}.
\end{align*}
Moreover, $P_q(1,n)=(1-q)^n$ for all $n$ and $P_q^+(1,n)=0$ if $n$ is odd and $P_q^+(1,n)=\binom{n}{n/2}q^{\frac{n}{2}}(1-q)^{\frac{n}{2}}$ if $n$ is even. Let $\phi_q(k,n)=q^k(1-q)^{n-k}\binom{n}{k}$, so that $k\mapsto \phi_q(k,n)$ is the probability density function of the binomial distribution with parameters $n$ and $q$.

The following asymptotic notation is used throughout the article. First, $\mathbb{N}$ is the symbol used for the set $\{1,2,3\dots\}$. If $f,g:\mathbb{N}\to\mathbb{R}_{>0}$ are functions, then $f(n)=O(g(n))$ means that there exist positive constants $C,N>0$ such that $f(n)\leq C g(n)$ for all $n\geq N$. The notation $f(n)=\Theta(g(n))$ means that there exist constants $C_1,C_2,N>0$ such that $C_1 g(n)\leq f(n)\leq C_2 g(n)$ for all $n\geq N$. Finally, $f(n)=o(g(n))$ is the notation used for $\lim_{n\to\infty}\frac{f(n)}{g(n)}=0$ and $f(n)\sim g(n)$ is the notation used for $\lim_{n\to\infty}\frac{f(n)}{g(n)}=1$, or equivalently, $f(n)=(1+o(1))g(n)$.

\mycomment{

The following lemma, proven in \cite{LL}, gives a description of the principal ideal $(\Phi_n(x))R_n$, i.e. the ideal of polynomials satisfying $f(x)\equiv 0\pmod{\Phi_n(x)}$, as a sum of simpler ideals.

\begin{lem}
    \label{lem:LL}
    Let $n=p_1^{a_1}\cdots p_s^{a_s}$ be the decomposition of $n$ into a product of primes. We have
    \begin{align*}
        (\Phi_n(x))R_n = \sum_{i=1}^{s} \left(\Phi_{p_i}(x^{\frac{n}{p_i}})\right) R_n.
    \end{align*}
    Further, if $s\leq 2$, then
    \begin{align*}
        (\Phi_n(x))S_n = \sum_{i=1}^{s} \left(\Phi_{p_i}(x^{\frac{n}{p_i}})\right) S_n,
    \end{align*}
    where $S_n\subseteq R_n$ denotes the semi-ring of polynomials with non-negative coefficients.
\end{lem}

Our goal in the coming sections is to investigate the behavior of $N(n):=|\mathcal{M}_n^0|$, or equivalently, the probability $\frac{N(n)}{2^n}$ that a uniformly randomly chosen binary circulant matrix is singular, as $n$ varies.

}

\section{Bounds on $P_q(d,n)$}
\label{sec:pq_bounds}
In this section, we derive upper and lower bounds on $P_q(d,n)$. The first key point towards our main theorem is the following simple observation. Let $d$ be a divisor of a positive integer $n$ and let $f(x)\in R_n$. Then $f(x) \equiv 0\pmod{\Phi_d(x)}$ if and only if $\overline{f}(x) \equiv 0\pmod{\Phi_d(x)}$ where $\overline{f}(x)$ denotes the image of $f(x)$ in the ring $R_d=\Z[x]/(x^d-1)$. Let $\overline{f}(x)=\sum_{j=0}^{d-1} s_j^{(d)} x^j$ be the image of $f(x)=\sum_{j=0}^{n-1} c_j x^j$ in $R_d$. Then the coefficients $s_j^{(d)}=\sum_{i\equiv j \pmod{d}} c_i$ of $\overline{f}$ are independent random variables which follow the binomial distribution with parameters $\frac{n}{d}$ and $q$. Hence, $\mathbb{P}[s_j^{(d)}=k]=\phi_q(k,\frac{n}{d})$ for all $j=0,1,\dots,d-1$. Since the $s_j^{(d)}$ are independent, 
\begin{align}
\label{eq:prob}
    P_q(d,n) &=\mathbb{P}[\overline{f}(x)\in \Phi_d(x)R_d]=\sum_{(s_0^{(d)},\dots,s_{d-1}^{(d)})\in \mathcal{S}_d}\prod_{j=0}^{d-1} \phi_q(s_j^{(d)},n/d),
\end{align}
where $\mathcal{S}_d$ denotes the subset of the $\Z$-module $M_d= \Phi_d(x) R_d\subseteq R_d\cong \mathbb{Z}^d$ consisting of vectors with coefficients in $[0,\frac{n}{d}]$ with respect to the standard basis. Note that $M_d$ has rank $d-\varphi(d)$, where $\varphi(d)$ is Euler's totient function.

If $d=p$ is prime, then by (\ref{eq:cyclo_prime}) $$\mathcal{S}_p=\left\{(a,a,\dots,a):0\leq a\leq \frac{n}{p}\right\},$$ and (\ref{eq:prob}) simplifies to
\begin{align}
\label{eq:prime}
    P_q(p,n)&=\sum_{k=0}^{\frac{n}{p}}\phi_q(k,n/p)^p.
\end{align}
More generally, when $d=p^m$ is a power of a prime, again by (\ref{eq:cyclo_prime}) $$\mathcal{S}_{p^{m}}=\left\{(b_0,b_1,\dots,b_{p^m-1}): 0\leq b_i\leq \frac{n}{p^{m}},\ b_i=b_j\text{ if $i\equiv j\pmod{p^{m-1}}$}\right\}.$$ Thus, (\ref{eq:prob}) simplifies to
\begin{align}
\label{eq:prime_power}
    P_q(p^m,n) &=\left(\sum_{k=0}^{\frac{n}{p^m}} \phi_q(k,n/p^m)^p\right)^{p^{m-1}}.
\end{align}
Using these identities and Lemma \ref{lem:cyclotomic_remainder}, one can obtain exact expressions for the probability $P_q(n)$, when $n$ is a small power of a prime or the product of two distinct primes.

\begin{lem}
    \label{lem:exact_probs}
    Let $p$ and $r$ be distinct prime numbers and $q\in(0,1)$. Then
    \begin{enumerate}
        \item $P_q(p)=q^p+(1-q)^{p}$,
        \item $P_q(p^2)=(q^p+(1-q)^p)^p+\sum_{k=0}^{p}\phi_q(k,p)^p-q^{p^2}-(1-q)^{p^2}$,
        \item $P_q(pr)=\sum_{k=0}^{p}\phi_q(k,p)^r+\sum_{k=0}^{r}\phi_q(k,r)^p-q^{pr}-(1-q)^{pr}$.
    \end{enumerate}
\end{lem}
\begin{proof} 
 Note that $\mathcal{P}_{1,n}\subseteq\mathcal{P}_{d,n}$ for all divisors $d\mid n$, and recall that $P_q(n)=\mathbb{P}[\cup_{d\mid n} \mathcal{P}_{d,n}]$. The first part follows directly from (\ref{eq:prime}). For the second part, we have
\begin{align*}
    P_q(p^2)&=P_q(p^2,p^2)+P_q(p,p^2)-\mathbb{P}[\mathcal{P}_{p,p^2}\cap\mathcal{P}_{p^2,p^2}].
\end{align*}
    The claim follows from (\ref{eq:prime_power}) and by noting that $f(x)\equiv 0\pmod{\Phi_{p^2}(x)\Phi_p(x)}$ if and only if $f(x)=a(1+x+\cdots+x^{p^2-1})$ for $a\in\{0,1\}$. For the third part, suppose that $f(x)\equiv 0\pmod{\Phi_{pr}(x)}$. Lemma \ref{lem:cyclotomic_remainder} implies that $\overline{f}(x)= \sum_{j=0}^{r-1} s_j^{(r)}x^j \equiv 0 \pmod{p,\Phi_r(x)}$, so $s_i^{(r)}\equiv s_j^{(r)}\pmod{p}$ for all $i,j$. Since $0\leq s_j^{(r)}\leq p$, we have $s_i^{(r)}=s_j^{(r)}$ for all $i,j$ or $s_i^{(r)}\in \{0,p\}$ for all $i$. The former case implies $f(x)\equiv 0\pmod{\Phi_r(x)}$ and the latter case implies $f(x)\equiv 0\pmod{\Phi_p(x)}$. Thus, $\mathcal{P}_{pr,pr}\subseteq \mathcal{P}_{p,pr}\cup\mathcal{P}_{r,pr}$, so
    \begin{align*}
        P_q(pr) &=P_q(p,pr)+P_q(r,pr)-\mathbb{P}[\mathcal{P}_{p,pr}\cap\mathcal{P}_{r,pr}].
    \end{align*}
    The claim follows from (\ref{eq:prime}) and the observation that $f(x)\equiv 0\pmod{\Phi_p(x)\Phi_r(x)}$ implies $f(1)\equiv 0\pmod{pr}$ so $f(x)=a(1+x+\cdots+x^{pr-1})$ where $a\in\{0,1\}$.
\end{proof}

Lemma \ref{lem:exact_probs} shows that the exact probabilities $P_q(n)$ become increasingly complicated as the number of divisors of $n$ grows. In fact, not much is understood about the coefficients of the $n$:th cyclotomic polynomials themselves when $n$ is the product of at least three distinct primes. Despite this, the individual probabilities $P_q(d,n)$ can be controlled well, using the next proposition.

\mycomment{
\begin{prop}
    \label{prop:binomial_max}
    Let $q\in (0,1)$ and $n\in\mathbb{N}$. There exists a constant $C_q$ depending only on $q$ such that
    \begin{align*}
        \frac{C_q}{\sqrt{2\pi n q(1-q)}} \leq \max_{0\leq k\leq n} q^{k}(1-q)^{n-k} \binom{n}{k} \leq \min\left\{\frac{1}{\sqrt{2\pi n q(1-q)}},\max\left\{ q,1-q\right\}\right\}.
    \end{align*}
\end{prop}
\begin{proof}
    By Stirling's approximation, $n!=\sqrt{2\pi}n^{n+\frac{1}{2}}e^{-n+\epsilon_n}$ for all positive integers $n$, where $\frac{1}{12n+1}<\epsilon_n<\frac{1}{12n}$. Thus,
    \begin{align*}
        q^k(1-q)^{n-k}\binom{n}{k}&=q^k(1-q)^{n-k}\frac{n!}{k!(n-k)!}e^{\epsilon_n-\epsilon_k-\epsilon_{n-k}}\\
        &=\frac{1}{\sqrt{2\pi nq(1-q)}}\frac{e^{\epsilon_n-\epsilon_k-\epsilon_{n-k}}}{\left(1-\frac{k-np}{nq}\right)^{n-k+\frac{1}{2}}\left(1+\frac{k-np}{np}\right)^{k+\frac{1}{2}}}.
    \end{align*}
\end{proof}
}

\begin{prop}
    \label{prop:Nd_bound}
    Let $n$ be a positive integer and $d\neq 1$ a divisor of $n$, and let $q\in (0,1)$. Let $M(q,n)=\max_{0\leq k \leq n} \phi_q(k,n)$. Then $P_q(d,n)\leq M(q,\frac{n}{d})^{\varphi(d)}$, where $\varphi$ is Euler's totient function. If $d$ is prime, then 
    $P_q(d,n)\geq M(q,\frac{n}{d})^d$.
\end{prop}
\begin{proof}
Note that $\{x^j \Phi_d(x):j=0,1,\dots,d-\varphi(d)-1\}$ constitutes a $\Z$-basis of $M_d=\Phi_d(x)R_d \subseteq R_d\cong \Z^d$. Therefore, since $\Phi_d(x)$ has constant coefficient one, $M_d$ (viewed as a sublattice of $\mathbb{Z}^d$) has a basis (the Hermite normal form) generated by the rows of the matrix $B=\left(I_{d-\varphi(d)}\mid A\right)$ for some matrix $A\in \Z^{(d-\varphi(d))\times\varphi(d)}$ where $I_{d-\varphi(d)}$ is the identity matrix. Each element $(s_0^{(d)},s_1^{(d)},\dots,s_{d-1}^{(d)})\in \mathcal{S}_d$ (where $\mathcal{S}_d$ is the subset of $M_d$ consisting of elements with coordinates in $[0,\frac{n}{d}]$) can be represented as
\begin{align*}
    (s_0^{(d)},s_1^{(d)},\dots,s_{d-1}^{(d)})=\sum_{i=0}^{d-\varphi(d)-1} z_i b_i,
\end{align*}
where $z_i\in \Z$ and $b_i$ denotes the $i$:th row of $B$. Since $0\leq s_j^{(d)}\leq \frac{n}{d}$, it is necessary (but not sufficient) that $0\leq z_i\leq \frac{n}{d}$ for all $i=0,1,\dots, d-\varphi(d)-1$. Thus, an upper bound on $P_q(d,n)$ is given by
\begin{align*}
P_q(d,n)&=\sum_{(s_0^{(d)},\dots,s_{d-1}^{(d)})\in \mathcal{S}_d}\prod_{j=0}^{d-1} \phi_q(s_j^{(d)},n/d)\\
    &\leq \sum_{0\leq z_0,\dots, z_{d-\varphi(d)-1}\leq \frac{n}{d}} M(q,n/d)^{\varphi(d)} \prod_{i=0}^{d-\varphi(d)-1} \phi_q(z_i,n/d) \\
    &= M(q,n/d)^{\varphi(d)}\left(\sum_{z=0}^{\frac{n}{d}}\phi_q(z,n/d)\right)^{d-\varphi(d)}= M(q,n/d)^{\varphi(d)}.
\end{align*}
The lower bound for $d$ prime follows from (\ref{eq:prime}) and 
\begin{align*}
    P_q(d,n)&=\sum_{k=0}^{\frac{n}{d}}\phi_q(k,n/d)^d\geq M(q,n/d)^d.
\end{align*}
\end{proof}

Note that the above lemma shows that $M(q,n/d)^{d} \leq P_q(d,n)\leq M(q,n/d)^{d-1}$ when $d$ is prime.

\section{Estimation of sums of powers of $\phi_q(k,n)$}
\label{sec:mq_bounds}

Now that we have developed expressions for the probabilities $P_q(d,n)$ and obtained bounds on them in terms of $M(q,n)$, we need to estimate $M(q,n)$ as well as a sum of the form $\sum_{k=0}^n \phi_q(k,n)^m$ where $m$ might depend on $n$. We begin with the former task. The following lemma will become useful.

\begin{lem}[de Moivre--Laplace]
    \label{lem:deMoivre}
    Let $q\in (0,1)$ and let $c>0$ be a real number. Then for all $k\in\mathbb{Z}$ with $|k-nq|\leq c\sqrt{n}$,
    \begin{align*}
        \phi_q(k,n)=\frac{1}{\sqrt{2\pi n q(1-q)}}e^{-\frac{(k-nq)^2}{2nq(1-q)}}\left(1+O\left(\frac{1}{\sqrt{n}}\right)\right),
    \end{align*}
 and for all $k\in\mathbb{Z}$ with $|k-nq|\leq c$,
 \begin{align*}
        \phi_q(k,n)=\frac{1}{\sqrt{2\pi n q(1-q)}}\left(1+O\left(\frac{1}{n}\right)\right),
    \end{align*}
    as $n\to\infty$.
\end{lem}
\begin{proof}
    See \cite[Theorem 2]{balazs2014stirling} for the first part. The second part follows from the proof of \cite[Theorem 2]{balazs2014stirling} after appropriate modifications.
\end{proof}

The maximum of the binomial distribution with parameters $n$ and $q\in (0,1)$ is attained at $k=\lfloor (n+1)q\rfloor$, and the value of this maximum is asymptotic to $M(q,n)\sim\frac{1}{\sqrt{2\pi n q(1-q)}}$ as $n\to \infty$, by Lemma \ref{lem:deMoivre}. The second part of Lemma \ref{lem:deMoivre} immediately gives the next lemma, which shows that if $m:\mathbb{N}\to\mathbb{N}$ is a function that does not grow too rapidly, then the above approximation remains valid after raising $M(q,n)$ to the $m(n)$:th power.

\begin{lem}
    \label{lem:mq_bounds}
    Let $q\in (0,1)$. Let $m:\mathbb{N}\to\mathbb{N}$ be a function such that $m(n)=O(n)$. Then $M(q,n)^{m(n)}=\Theta\left((2\pi n q(1-q))^{-\frac{m(n)}{2}}\right)$ as $n\to\infty$.
\end{lem}
\begin{proof}
    By Lemma \ref{lem:deMoivre}, there exist constants $C_0,N_0>0$ with $\frac{C_0}{N_0}<1$, and where $C_0$, $N_0$ only depend on $q$, such that $\frac{1}{\sqrt{2\pi nq(1-q)}}(1-\frac{C_0}{n})\leq M(q,n)\leq \frac{1}{\sqrt{2\pi nq(1-q)}}(1+\frac{C_0}{n})$ for all $n>N_0$. Since $m(n)=O(n)$, there exist constants $C_1,N_1>N_0$ such that
    \begin{align*}
        M(q,n)^{m(n)} &\leq (2\pi nq(1-q))^{-\frac{m(n)}{2}}\left(1+\frac{C_0}{n}\right)^{m(n)}\\
        &\leq (2\pi nq(1-q))^{-\frac{m(n)}{2}}\left(1+\frac{C_0}{n}\right)^{C_1 n}\leq (2\pi nq(1-q))^{-\frac{m(n)}{2}}e^{C_0 C_1}
    \end{align*}
    and
    \begin{align*}
        M(q,n)^{m(n)}&\geq (2\pi nq(1-q))^{-\frac{m(n)}{2}}\left(1-\frac{C_0}{n}\right)^{m(n)}\\
        &\geq (2\pi n q(1-q))^{-\frac{m(n)}{2}} \left(1-\frac{C_0}{n}\right)^{C_1 n}\geq (2\pi n q(1-q))^{-\frac{m(n)}{2}} \cdot \frac{1}{2}e^{-C_0 C_1}
    \end{align*}
    for all $n>N_1$.
\end{proof}
Now we turn to the next task of this section. The implicit constants in the next proposition only depend upon $q$.

\begin{prop}
\label{prop:asymptotic2}
    Let $q\in (0,1)$ be fixed. Let $m:\mathbb{N}\to\mathbb{N}$ be a function such that $m(n)=O(\sqrt{n})$. Then
    \begin{align*}
         \sum_{k=0}^n \phi_q(k,n)^{m(n)} =\Theta\left( \frac{1}{(2\pi q(1-q)n)^{\frac{m(n)-1}{2}}\sqrt{m(n)}}\right).
    \end{align*}
    If furthermore $m(n)=o(\sqrt{n})$, then
    \begin{align*}
         \sum_{k=0}^n \phi_q(k,n)^{m(n)} \sim \frac{1}{(2\pi q(1-q)n)^{\frac{m(n)-1}{2}}\sqrt{m(n)}}.
    \end{align*}
\end{prop}

\begin{proof}
    For a proof of the asymptotic formula in the special case $q=\frac{1}{2}$ and $m(n)$ fixed, see \cite{farmer2004asymptotic}. Let $0<\epsilon<1$ be given, and let $c>0$ be a real number such that $\frac{1}{\sqrt{2\pi}}\int_{-c}^c e^{-\frac{1}{2}t^2}dt\geq 1-\epsilon$. Let $I_{c,n}=\{k\in\Z:|k-nq|\leq c\sqrt{nq(1-q)}\}$. For ease of notation, let $K_q=2\pi q(1-q)$. By Lemma \ref{lem:deMoivre}, there exist constants $C_0,N_0>0$ (with $\frac{C_0}{\sqrt{N_0}}<1$), such that for all $n>N_0$ and $k\in I_{c,n}$,

    \begin{align*}
        \left(1-\frac{C_0}{\sqrt{n}}\right)\frac{1}{\sqrt{K_q n}}e^{-\frac{(k-nq)^2}{2nq(1-q)}} &\leq  \phi_q(k,n) \leq \left(1+\frac{C_0}{\sqrt{n}}\right) \frac{1}{\sqrt{K_q n}}e^{-\frac{(k-nq)^2}{2nq(1-q)}}.
    \end{align*}

   Thus, for these values of $n$ and $k$,

   \begin{align*}
        &\left(1-\frac{C_0}{\sqrt{n}}\right)^{m(n)}\frac{1}{(K_q n)^{\frac{m(n)}{2}}}e^{-\frac{m(n)(k-nq)^2}{2nq(1-q)}} \\
        &\leq \phi_q(k,n)^{m(n)}\leq \left(1+\frac{C_0}{\sqrt{n}}\right)^{m(n)} \frac{1}{(K_qn)^{\frac{m(n)}{2}}}e^{-\frac{m(n)(k-nq)^2}{2nq(1-q)}}.
    \end{align*}

   Since $m(n)=O(\sqrt{n})$, analogously to the proof of Lemma \ref{lem:mq_bounds}, there exist positive constants $C_l,C_u$ and $N_1>N_0$ such that for all $n>N_1$ and $k\in I_{c,n}$,

  \begin{align*}
        \frac{C_l}{(K_q n)^{\frac{m(n)}{2}}}e^{-\frac{m(n)(k-nq)^2}{2nq(1-q)}} \leq  \phi_q(k,n)^{m(n)} \leq \frac{C_u}{(K_q n)^{\frac{m(n)}{2}}}e^{-\frac{m(n)(k-nq)^2}{2nq(1-q)}}.
    \end{align*}
  
   If $m(n)=o(\sqrt{n})$, then $\left(1\pm\frac{C_0}{\sqrt{n}}\right)^{m(n)}\to1$ as $n\to \infty$, so we may take $C_u=1+\epsilon$ and $C_l=1-\epsilon$. Next we claim that 
   
   \begin{align*}
       \sum_{k\in I_{c,n}} e^{-\frac{m(n)(k-nq)^2}{2nq(1-q)}}&=(1+o(1))\int_{-c\sqrt{nq(1-q)}}^{c\sqrt{nq(1-q)}} e^{-\frac{m(n)t^2}{2nq(1-q)}}dt\\
       &=(1+o(1))\sqrt{nq(1-q)}\int_{-c}^{c} e^{-\frac{m(n) t^2}{2}} dt
   \end{align*}
   where the latter equality comes from the substitution $t=\sqrt{nq(1-q)} u$. Consider the function $f(t)=e^{-at^2}$, where $a>0$. The integral $\int_{-b}^{b}f(t)dt$ can be approximated by rectangles below and above, which gives the bound
   \begin{align*}
       \left|\sum_{t=-b}^{b} f(t)-\int_{-b}^{b}f(t)dt\right|<f(0)=1
   \end{align*}
    for all $b>0$, where the summation runs over all integers in the interval $[-b,b]$. Therefore, $|\sum_{k\in I_{c,n}} e^{-\frac{m(n)(k-nq)^2}{2nq(1-q)}}-\int_{-c\sqrt{nq(1-q)}}^{c\sqrt{nq(1-q)}} e^{-\frac{m(n)t^2}{2nq(1-q)}}dt|<1$ for all $n$. Note that $\sqrt{\frac{m(n)}{2\pi}}\int_{-c}^{c} e^{-\frac{m(n)}{2}t^2}dt\geq \frac{1}{\sqrt{2\pi}}\int_{-c}^{c}e^{-\frac{1}{2}t^2}dt\geq  1-\epsilon$. Therefore, 
    \begin{align*}
      \sqrt{nq(1-q)}\int_{-c}^{c}e^{-\frac{m(n)}{2}t^2}dt\geq \sqrt{nq(1-q)}\sqrt{\frac{2\pi}{m(n)}}(1-\epsilon)\to\infty   
    \end{align*}
    as $n\to\infty$. This proves the claim.  Thus, we can find a constant $N_2>N_1$ such that for all $n>N_2$,

    \begin{align*}
        &\frac{(1-\epsilon)C_l}{(K_q n)^{\frac{m(n)}{2}}}\sqrt{nq(1-q)}\int_{-c}^{c} e^{-\frac{m(n) t^2}{2}} dt\\
        &\leq  \sum_{k\in I_{c,n}} \phi_q(k,n)^{m(n)}\leq \frac{(1+\epsilon)C_u}{(K_qn)^{\frac{m(n)}{2}}}\sqrt{nq(1-q)}\int_{-c}^{c} e^{-\frac{m(n) t^2}{2}} dt.
    \end{align*}

   In particular, all of the above holds when $m(n)\equiv 1$ (with $C_l=1-\epsilon$), so we may find a constant $N_3>N_2$ such that $\sum_{k\in I_{c,n}}\phi_q(k,n)\geq\frac{(1-\epsilon)^2}{\sqrt{2\pi}}\int_{-c}^{c}e^{-\frac{t^2}{2}}dt$ for all $n>N_3$. Using the fact that $\phi_q(k,n)\geq \phi_q(k',n)$ for all $k\in I_{c,n}$ and $k'\in [0,n]\setminus I_{c,n}$ we obtain
   \begin{align*}
       1&\geq \frac{\sum_{k\in I_{c,n}}\phi_q(k,n)^{m(n)}}{\sum_{k=0}^{n}\phi_q(k,n)^{m(n)}}\geq \frac{\sum_{k\in I_{c,n}}\phi_q(k,n)}{\sum_{k=0}^{n}\phi_q(k,n)}=\sum_{k\in I_{c,n}}\phi_q(k,n)\\
       &\geq (1-\epsilon)^2\frac{1}{\sqrt{2\pi}}\int_{-c}^{c}e^{-\frac{t^2}{2}}dt\geq (1-\epsilon)^3,
   \end{align*}
    for all $n>N_3$. For these values of $n$,
    
    \begin{align*}
        &\sum_{k=0}^{n}\phi_q(k,n)^{m(n)} \leq \frac{1}{(1-\epsilon)^3} \sum_{k\in I_{c,n}}\phi_q(k,n)^{m(n)}\\
        &\leq \frac{(1+\epsilon)C_u}{(1-\epsilon)^3 (K_q n)^{\frac{m(n)}{2}}}\sqrt{nq(1-q)}\int_{-\infty}^{\infty} e^{-\frac{m(n) t^2}{2}} dt.
    \end{align*}
    
Using again that $\sqrt{\frac{m(n)}{2\pi}}\int_{-c}^{c} e^{-\frac{m(n)}{2}t^2}dt\geq \frac{1}{\sqrt{2\pi}}\int_{-c}^{c}e^{-\frac{1}{2}t^2}dt\geq  1-\epsilon$ for all $n$, we get for $n>N_3$,
\begin{align*}
    &\sum_{k=0}^{n}\phi_q(k,n)^{m(n)} \geq \sum_{k\in I_{c,n}}\phi_q(k,n)^{m(n)}\\
    &\geq\frac{(1-\epsilon)C_l}{(K_q n)^{\frac{m(n)}{2}}}\sqrt{nq(1-q)}\int_{-c}^{c} e^{-\frac{m(n) t^2}{2}} dt\\
    &\geq\frac{(1-\epsilon)^2C_l}{(K_q n)^{\frac{m(n)}{2}}}\sqrt{nq(1-q)}\int_{-\infty}^{\infty} e^{-\frac{m(n) t^2}{2}} dt.
\end{align*}

Finally, noting that $\sqrt{\frac{m(n)}{2\pi}}\int_{-\infty}^{\infty} e^{-\frac{m(n)}{2}t^2}dt=1$ and recalling that $C_l=1-\epsilon$ and $C_u=1+\epsilon$ if $m(n)=o(\sqrt{n})$, the asymptotic formulas follow.

\end{proof}

\section{Main theorem}
\label{sec:main}
With the results of the previous sections we can now prove the main theorem of the article (Theorem \ref{thm:main0}). Recall that $p(n)$ denotes the smallest prime divisor of $n>1$. All constants involved in the proof of the main theorem only depend on the parameter $q$.

\begin{proof}
    First note that $\mathcal{P}_{1,n}\subseteq\mathcal{P}_{d,n}$ for all $d\mid n$. Therefore, by the union bound, $ P_q(p(n),n)\leq P_q(n)\leq \sum_{d\mid n,d\neq 1} P_q(d,n)$ for all $n$. By Lemma \ref{lem:exact_probs} (part 1), it suffices to show that $\frac{\sum_{d\mid n,d\neq 1,p(n)} P_q(d,n)}{P_q(p(n),n)}\to 0$ as $n\to \infty$ for composite $n$, so from now on we only consider composite $n$. Let $0<\epsilon,\delta<10^{-3}$ (say). We will need the fact that $\tau(n)<n^{\epsilon}$ and $n^{1-\delta}<\varphi(n)<n$ for all sufficiently large $n$, where $\tau(n)$ is the number of divisors of $n$ and $\varphi(n)$ is Euler's totient function. We will also use the fact that $p(n)\leq \sqrt{n}$.

    The divisors $d\neq 1,p(n)$ of $n$ are partitioned into four sets (some of which might be empty):
    \begin{align*}
        S_0&=\{d\mid n: n^{\frac{1}{2}+\delta}\leq d \leq n\}, \\
        S_1&=\{d\mid n: 10p(n)\leq d<n^{\frac{1}{2}+\delta}\}, \\
        S_2&=\{d\mid n:p(n)<d<10p(n), \text{ $d$ prime, $d\neq 3$}\}, \\
        S_3&=\{d\mid n:p(n)<d<10p(n), \text{ $d$ composite or $d=3$}\}.
    \end{align*}
    Let $K_q=2\pi q(1-q)$. By Proposition \ref{prop:Nd_bound} and Lemma \ref{lem:mq_bounds}, we can find a constant $C>0$ such that $P_q(p(n),n)\geq M(q,n/p(n))^{p(n)}\geq C\left(\frac{K_q n}{p(n)}\right)^{-\frac{p(n)}{2}}$ for all large enough $n$. 
    
    For each $d\in S_0$, it will be enough to use the trivial bound $P_q(d,n) \leq M(q,\frac{n}{d})^{\varphi(d)}\leq m_q^{\varphi(d)}$, where $m_q=\max\{q,1-q\}<1$. Thus, recalling that $|S_j|\leq \tau(n)< n^{\epsilon}$ and $\varphi(d)>n^{\frac{1}{2}+\frac{\delta}{2}}$ for $d\in S_0$ and $n$ large,
    \begin{align*}
        \frac{\sum_{d\in S_0} P_q(d,n)}{P_q(p(n),n)} &\leq C^{-1}\left(\frac{K_q n}{p(n)}\right)^{\frac{p(n)}{2}}\sum_{d\in S_0} m_q^{\varphi(d)} \leq C^{-1} \left(K_q n\right)^{n^{\frac{1}{2}}}n^{\epsilon}m_q^{n^{\frac{1}{2}+\frac{\delta}{2}}}\to 0
    \end{align*}
   as $n\to \infty$.

   Now consider $d\in S_1$. Then, by Lemma \ref{lem:mq_bounds} we can find a constant $C_0>0$ such that $M(q,n/d)\leq C_0(\frac{K_q n}{d})^{-\frac{1}{2}}$. Therefore, $P_q(d,n)\leq M(q,n/d)^{\varphi(d)}\leq \left(\frac{C_0^2 d}{K_q n}\right)^{\frac{\varphi(d)}{2}}=\left(\frac{C_1 d}{n}\right)^{\frac{\varphi(d)}{2}}$, where $C_1=\frac{C_0^2}{K_q}$.  Noting that $\varphi(d)\geq 3p(n)$ for all $d\in S_1$, we get for sufficiently large $n$,
   \begin{align*}
       \frac{\sum_{d\in S_1} P_q(d,n)}{P_q(p(n),n)} &\leq C^{-1} \left(\frac{K_q n}{p(n)}\right)^{\frac{p(n)}{2}}  \sum_{d\in S_1}\left(\frac{C_1 d}{n}\right)^{\frac{\varphi(d)}{2}} \\
       &\leq C^{-1} \left(\frac{K_q n}{p(n)}\right)^{\frac{p(n)}{2}} \sum_{d\in S_1} \left(\frac{C_1 d}{n}\right)^{\frac{3p(n)}{2}}\\
       &= C^{-1}\sum_{d\in S_1}\left(\frac{K_q C_1^3 d^3}{p(n) n^2}\right)^{\frac{p(n)}{2}}\leq C^{-1}n^{\epsilon}\left(\frac{K_qC_1^3 n^{\frac{3}{2}+3\delta}}{n^2}\right)^{\frac{p(n)}{2}}\\
       &= C^{-1}n^{\epsilon}\left(\frac{K_qC_1^3}{n^{\frac{1}{2}-3\delta}}\right)^{\frac{p(n)}{2}}\leq C^{-1}n^{\epsilon}\frac{K_qC_1^3}{n^{\frac{1}{2}-3\delta}} \to 0.
   \end{align*}
   Suppose that $d\in S_2$. In this case $\varphi(d)\leq d\leq 100(\frac{n}{d})$, since $d<10p(n)\leq10\sqrt{n}$. Thus, Lemma \ref{lem:mq_bounds} gives $P_q(d,n)\leq M(q,\frac{n}{d})^{\varphi(d)}\leq C_2\left(\frac{K_q n}{d}\right)^{-\frac{\varphi(d)}{2}}$ for some constant $C_2>0$ and all $d\in S_2$. Since $d$ is a prime different from 3, $\varphi(d)=d-1>p(n)$. Therefore, for sufficiently large $n$,
   \begin{align*}
       \frac{\sum_{d\in S_2} P_q(d,n)}{P_q(p(n),n)} &\leq C^{-1}C_2\sum_{d\in S_2} \left(\frac{K_qn}{p(n)}\right)^{\frac{p(n)}{2}}\left(\frac{d}{K_q n}\right)^{\frac{\varphi(d)}{2}}\\
       &=C^{-1}C_2\sum_{d\in S_2} \left(\frac{d}{p(n)}\right)^{\frac{p(n)}{2}}\left(\frac{d}{K_qn}\right)^{\frac{\varphi(d)-p(n)}{2}}\\
       &=C^{-1}C_2 \sum_{d\in S_2}  \left(1+\frac{d-p(n)}{p(n)}\right)^{\frac{p(n)}{2}}\left(\frac{d}{K_qn}\right)^{\frac{\varphi(d)-p(n)}{2}}\\
       &\leq C^{-1}C_2\sum_{d\in S_2}  e^{\frac{d-p(n)}{2}}\left(\frac{d}{K_qn}\right)^{\frac{\varphi(d)-p(n)}{2}}\\
       &=C^{-1}C_2e^{\frac{1}{2}}\sum_{d\in S_2} \left(\frac{ed}{K_qn}\right)^{\frac{\varphi(d)-p(n)}{2}}\\
       &\leq C^{-1}C_2n^{\epsilon}e^{\frac{1}{2}}\left(\frac{10e}{K_qn^{\frac{1}{2}}}\right)^{\frac{1}{2}}\to 0.
   \end{align*}
  The final case $d\in S_3$ requires a tighter lower bound on $P_q(p(n),n)$. Note that $d\in S_3$ implies  that either $p(n)=2$ and $d=3$, or $d$ is composite so $p(n)^2\leq d\leq 10p(n)$. In either case, $d$ and $p(n)$ are bounded by 100. Therefore, we may use Proposition \ref{prop:asymptotic2}: there is a constant $C_3>0$ such that for $n$ large enough, $P_q(p(n),n)=\sum_{k=0}^{\frac{n}{p(n)}}\phi_q(k,n/p(n))^{p(n)}\geq C_3 p(n)^{-\frac{1}{2}}\left(\frac{K_qn}{p(n)}\right)^{\frac{1-p(n)}{2}}$. Therefore, if $n$ is sufficiently large, then
  \begin{align*}
      \frac{\sum_{d\in S_3}P_q(d,n)}{P_q(p(n),n)}  &\leq C_3^{-1}C_2\sum_{d\in S_3} p(n)^{\frac{1}{2}}\left(\frac{K_qn}{p(n)}\right)^{\frac{p(n)-1}{2}}\left(\frac{d}{K_qn}\right)^{\frac{\varphi(d)}{2}}\\
      &\leq C_4 \sum_{d\in S_3}n^{\frac{(p(n)-1)-\varphi(d)}{2}}\leq \frac{C_5}{n^{\frac{1}{2}}}\to 0,
  \end{align*}
  for some constants $C_4,C_5>0$, where we used that $d$, $p(n)$ and the number of summands are bounded, and that $\varphi(d)>p(n)-1$. This finishes the proof of the theorem.
\end{proof}


We obtain the following asymptotic expression for $P_q^{+}(n)$.

\begin{cor}
    \label{cor:main}
    Let $q\in (0,1)$ be fixed. Then
    \begin{align*}
        P_q^{+}(n) \sim \sum_{k=0}^{\frac{n}{p(n)}} \phi_q(k,n/p(n))^{p(n)},
    \end{align*}
    except if $q=\frac{1}{2}$ and $n$ is even, in which case
    \begin{align*}
    P_q^{+}(n) &\sim \frac{2\sqrt{2}}{\sqrt{\pi n}}
    \end{align*}
    as $n\to\infty$.
\end{cor}
\begin{proof}
    Recall that $P_q^{+}(d,n)=P_q(d,n)$ for all divisors $d\neq 1$ of $n$. Furthermore, $P_q^{+}(1,n)=\binom{n}{n/2}q^{n/2}(1-q)^{n/2}$ if $n$ is even and $0$ if $n$ is odd. Therefore, if $n$ is odd the claim follows immediately from Theorem \ref{thm:main0}. So suppose that $n$ is even. Then 
    \begin{align*}
        &P_q^{+}(1,n)+P_q^{+}(2,n)-\mathbb{P}[\mathcal{P}_{1,n}^+\cap\mathcal{P}_{2,n}^{+}]\leq P_q^{+}(n)\leq \sum_{d\mid n}P_q^{+}(d,n).
    \end{align*}
    By Stirling's approximation, Proposition \ref{prop:asymptotic2}, and a simple calculation,
    \begin{align*}
    P_q^{+}(1,n) &\sim \frac{2^n}{\sqrt{\pi n/2}}q^{n/2}(1-q)^{n/2}=\frac{(2\sqrt{q(1-q)})^n}{\sqrt{\pi n/2}},\\
        P_q^{+}(2,n) &= \sum_{k=0}^{\frac{n}{2}}\phi_q(k,n/2)^2\sim \frac{1}{\sqrt{2\pi q(1-q)n}},\\
        \mathbb{P}[\mathcal{P}_{1,n}^{+}\cap\mathcal{P}_{2,n}^{+}] &=\begin{cases}
            0,&n\not\equiv 0\pmod{4}\\
            \binom{n/2}{n/4}^2 q^{n/2}(1-q)^{n/2},& n\equiv 0\pmod{4}
        \end{cases}\\
        &\sim\begin{cases}
            0,&n\not\equiv 0\pmod{4}\\
            \frac{(2\sqrt{q(1-q)})^n}{\pi n/4},& n\equiv 0\pmod{4}.
        \end{cases}
    \end{align*}
   If $q\neq\frac{1}{2}$, then $\frac{P_q^{+}(1,n)}{P_q^{+}(2,n)}\to 0$ since $2\sqrt{q(1-q)}<1$. Thus, the claim follows from the proof of Theorem \ref{thm:main0} and $\mathbb{P}[\mathcal{P}_{1,n}^{+}\cap\mathcal{P}_{2,n}^{+}]\leq P_q^{+}(1,n)$. If $q=\frac{1}{2}$, then $P_q^{+}(1,n)=P_q^{+}(2,n)$ for all $n$. By the above, $\frac{\mathbb{P}[\mathcal{P}_{1,n}^{+}\cap\mathcal{P}_{2,n}^{+}]}{P_q^{+}(1,n)}\to 0$ as $n\to\infty$. Thus,
   \begin{align*}
       P_{q}^{+}(n)\sim 2P_q^{+}(1,n)\sim \frac{2\sqrt{2}}{\sqrt{\pi n}}.
   \end{align*}
\end{proof}

\bibliographystyle{siam}
\bibliography{refs}

\end{document}